\documentclass{amsart}
\usepackage{amssymb}
\usepackage{subfigure}

\usepackage{graphicx}
\usepackage[all]{xy}
\usepackage{color}

\newtheorem{defn}{Definition}[section]
\newtheorem{thm}[defn]{Theorem}
\newtheorem{lem}[defn]{Lemma}

\newtheorem{exam}[defn]{Example}

\newcommand\B{\mathbf{B}}
\author{Byung Hee An}
\address{Department of Mathematics, POSTECH, Pohang, 790-784, Korea}
\email{anbyhee@postech.ac.kr}
\author{Ki Hyoung Ko}
\address{Department of Mathematics, KAIST, Daejeon, 307-701, Korea}
\email{knot@kaist.ac.kr}
\thanks{This work was supported by Priority Research Centers Program through the National Research Foundation of Korea (NRF) funded by the Ministry of Education, Science and Technology (2009-0094069)}

\begin{document}
\title[]{A family of pseudo-Anosov braids with large conjugacy invariant sets}
\begin{abstract}
We show that there is a family of pseudo-Anosov braids independently parameterized by the braid index and the (canonical) length whose smallest conjugacy invariant sets grow exponentially in the braid index and linearly in the length and conclude that the conjugacy problem remains exponential in the braid index under the current knowledge.
\end{abstract}
\maketitle

\section{Introduction and Preliminaries}

We use the following Artin's presentation \cite{Art} of the group $\B_n$ of $n$-braids:
$$
\B_n=\left\langle
\sigma_1,\dots,\sigma_{n-1}\big\vert
\begin{array}{ll}
\sigma_i\sigma_j=\sigma_j\sigma_i & \text{if }|i-j|>1\\
\sigma_i\sigma_j\sigma_i=\sigma_j\sigma_i\sigma_j & \text{if }|i-j|=1
\end{array}
\right\rangle,
$$
An $n$-braid in the monoid $B_n^+$ of positive words is a {\em permutation} braid if a pair of $n$-strands crosses at most once. The set $\mathbf S_n$ of
permutation $n$-braids is in one-to-one correspondence with the permutation group of $n$ elements. For $\beta\in\mathbf{S}_n$, the {\em starting set $S(\beta)$} and the {\em finishing set $F(\beta)$} is given by
$$
S(\beta)=\left\{ i | \beta=\sigma_i\beta' \text{ for some }\beta'\in\mathbf{S}_n\right\},
F(\beta)=\left\{ i | \beta=\beta'\sigma_i \text{ for some }\beta'\in\mathbf{S}_n\right\}.
$$
The half-twist braid $\Delta$ is the only braid in $\mathbf S_n$ whose starting and finishing sets are the largest set $\{1,\ldots,n-1\}$. Every $n$-braid can be put into its {\em left canonical form}, that is,
$$\beta=\Delta^k \beta_1\cdots \beta_\ell$$
where $\beta_i\in \mathbf S_n$ and $F(\beta_i)\supset S(\beta_{i+1})$ for all $i=1,\ldots,\ell-1$. The integers $\ell$, $k$, and $k+\ell$ are called the {\em canonical length}, the {\em infimum}, and the {\em supremum} of  $\beta$, respectively. This decomposition introduced by Garside \cite{Gar} immediately solves the word problem of $B_n$ and Garside showed that the {\em summit set}
$SS(\beta)$, the set of all braids with the maximal infimum and conjugate to a $\beta\in\B_n$, is finite to give a theoretical solution to the conjugacy problem.

Since then, several smaller
sets that are invariant under conjugate have been proposed.
Elrifai and Morton \cite{EM} studied {\em super summit
set $SSS(\beta)$}, the set of all braids with the maximal infimum and the minimal supremum and conjugate to $\beta$, and
proved that any braid can be pushed into its super summit set
through iteration of two natural operations on a braid $\beta=\Delta^k \beta_1\cdots \beta_\ell$ given by its canonical form. The cycling of $\beta$ is given by
$$\mathbf c(\beta)=\Delta^k \beta_2\cdots \beta_\ell\tau^k(\beta_1)$$
and the decycling of $\beta$ is given by
$$\mathbf d(\beta)=\Delta^k \tau^k(\beta_\ell)\beta_1\cdots \beta_{\ell-1}$$
where $\tau:\B_n\to\B_n$ be the involution defined by $\tau(\beta)=\Delta^{-1}\beta\Delta$.
They also introduce a partial order $\prec$ among positive $n$-braids such that $\alpha\prec\beta$ if $\alpha$ is a prefix of $\beta$, that is, $\alpha\gamma=\beta$ for some $\gamma\in B_n^+$.
In his thesis, Lee consider the {\em reduced super summit set $RSSS(\beta)$} of all braids conjugate to $\beta$ and fixed by iterated cycling and decycling. On the other hand, Franco and Gonz\'alez-Meneses \cite{FG} gave
a practical algorithm to generate $SSS(\beta)$ by using the following lattice structure on $S_n$ with respect to the partial order $\prec$. Given $\gamma_1, \gamma_2\in S_n$, the {\em meet} $\gamma_1\wedge\gamma_2$ is the maximal common prefix of $\gamma_1$ and $\gamma_2$ and the {\em join} $\gamma_1\vee\gamma_2$ is the minimal braid that has prefixes $\gamma_1$ and $\gamma_2$, and two operations can be extended over $B_n^+$. In fact, they showed that the meet operation $\wedge$ is closed among conjugators of $\beta$ that produce elements in $SSS(\beta)$ and so there are minimal conjugators moving around $SSS(\beta)$.
Gebhardt \cite{Geb} proposed the {\em
ultra summit set $USS(\beta)$} that collects all braids in $SSS(\beta)$ and fixed by iterated cycling and proved that the algorithm using
minimal conjugators also works to generate $USS(\beta)$. Finally, Gebhardt and Gonz\'alez-Meneses \cite{GG, GG2} introduced a new conjugation operation called a {\em cyclic sliding} and considered the set $SC(\beta)$, called the {sliding circuit}, of orbits of $\beta$ under the operation. These finite conjugacy invariant sets of a given braid $\beta$ in general satisfy
$$
SC(\beta)\subset RSSS(\beta)\subset USS(\beta)\subset SSS(\beta)\subset SS(\beta).
$$

A braid $\beta=\Delta^k A_1\cdots A_\ell$ given in canonical form is {\em rigid} if $\ell\ge 2$ and $F(A_1)\supset S(\tau^k(A_{\ell}))$. If a braid $\beta$ is conjugate to a rigid braid $\alpha$, then $\alpha\in RSSS(\beta)$ and moreover it was shown in \cite{BGG1,GG, KL} that  $SC(\beta)=RSSS(\beta)=USS(\beta)=R(\beta)$ where $R(\beta)$ denotes the set of all rigid braids conjugate to $\beta$. When the canonical length $\ell$ is 1, cycling and decycling is meaningless. If we assume that the rigidity requires $F(A_1)\supset S(\tau^k(A_{1}))$ in this case, one can show that $SC(\beta)=SSS(\beta)=R(\beta)$. For a rigid braid $\beta$, $R(\beta)$ is the smallest among invariant conjugacy sets of $\beta$ under current knowledge and will be called the {\em rigid conjugacy set} of $\beta$.

As a self-homeomorphism on $n$-punctured 2-disk, an $n$-braid
$\beta$ is of one of the following three dynamic types known as {\em
Nielsen-Thurston classification} \cite{Thu1}: (1)
 $\beta$ is {\em periodic} if some power of $\beta$ is isotopic to
some power of Dehn twist along the boundary of disk; (2) $\beta$ is
{\em reducible} if there exists a family of pairwise disjoint
essential non-peripheral simple closed curves in $D_n$, called {\em
reduction system}, which is preserved by a $n$-braid isotopic to
$\beta$; (3) $\beta$ is {\em pseudo-Anosov} is neither (1) nor (2).

The conjugacy problem for periodic braids is relatively easy \cite{Gon, BGG3}.
For reducible braids, once the reduction system is known, the conjugacy problem can be reduced into several pieces, which are for either periodic or pseudo-Anosov braids. It is reasonable to say that the conjugagacy problem for pseudo-Anosov braids is essential.

If a reduction system consists of circles that intersect the axis containing all punctures exactly twice, then it is called a {\em standard reduction system}. Bernardete, Nitecki and GuitierrezIt \cite{BNG} showed that cycling and decycling operations retain a standard reduction system and so $RSSS(\beta)$ must contain a braid with a standard reduction system if $\beta$ is reducible. This also implies that if a braid has a standard reduction system then each factor of its canonical form retain a standard reduction system and so a standard reduction system can be detected in polynomial time \cite{LL}.

For pseudo-Anosov braids, Masur and Minsky \cite{MM1,MM2} showed that there exists a constant $K$ which depends only on the braid index $n$ such that
for any conjugate $\alpha,\beta\in\B_n$, there exists a conjugator $\gamma$ satisfying
$|\gamma|\le K(|\alpha|+|\beta|)$.
Even though their proof leaves no clue how to find such $\gamma$, the conjugacy problem for pseudo-Anosov braids seems to have a low complexity in length.
On the other hand, Ko and Lee in \cite{KL} and Birman, Gebhardt and Gonz\'alez-Meneses in \cite{BGG1} independently proved that a pseudo-Anosov braid becomes rigid up to taking power and cycling. We remark that the exponent of the power needed has an upper bound that is a polynomial only in the braid index but the number of cycling required is unknown even though it is relatively small on computer experiment.
Consequently, the conjugacy problem for pseudo-Anosov braids boils down to that for rigid pseudo-Anosov braids. For rigid braids, the complexity of the current best solution ultimately relies on the size of the rigid conjugacy set of an input braid.

In this article, we show that there is a family of pseudo-Anosov braids independently parameterized by the braid index and the (canonical) length whose rigid conjugacy sets grow exponentially in the braid index and linearly in the length and conclude that the conjugacy problem remains exponential in the braid index under current knowledge. Our strategy is to construct a family of rigid $n$-braids of canonical length $k$ from the set of $k\times\lfloor \frac n2\rfloor$ binary matrices. Then we show that each member of the family carries no standard reduction system and the rigid conjugacy set of each member is a subset of the family and therefore each member is pseudo-Anosov. Finally we compute the size of the rigid conjugacy set of each member and conclude that the size of the rigid conjugacy set of a well chosen member is still exponential in $n$.

There have been some results relevant to our work. Prasolov \cite{Pr} constructed a family of rigid pseudo-Anosov $n$-braids of canonical length $n$ whose rigid conjugacy sets grow exponentially in $n$. This family obviously has the limitation that the canonical length depends on the braid index. On the other hand, Gonz\'alez-Meneses in \cite{Gon2} constructed a family of reducible $n$-braids of canonical length $k$ whose sliding circuits grow exponentially in both $n$ and $k$.

\section{Braids corresponding to binary matrix}
Let $\left\{ 0,1 \right\}^p$ be the set of binary $p$-tuples and $\mathbf{a}\in\left\{ 0,1 \right\}^p$.
Using the four building blocks $\alpha_I, \alpha_0, \alpha_1$ and $\alpha_T$ depicted in Figure~\ref{fig:buildingblocks}, we construct a braid $\alpha(\mathbf a)$ for a given $\mathbf a=(a_1,\ldots,a_p)\in \{ 0,1 \}^p$ by gluing adjacent blocks together by identifying circled crossings in the row $\alpha_I,\alpha_{a_1},\ldots,\alpha_{a_p},\alpha_T$ of building blocks.
Then $\alpha(\mathbf{a})$ is a positive $(2p+2)$-braid of $4p+1$ crossings in which every pair of strands crosses each other at most once. This can be seen, for example, by the fact that every strand can be drawn by a straight line. Thus $\alpha(\mathbf{a})$ is a permutation braid and we obtain a function $\alpha:\left\{ 0,1 \right\}^p\to\mathbf{S}_{2p+2}$. Figure~\ref{fig:Examplebinarybraid} shows the braid $\alpha(0,1,1,0,0,1)$.

For any permutation braid $\beta$, we consider the induced permutation $\pi(\beta)$ defined by $\pi(\beta)(i)=j$ when $i$-th strand in $\beta$ ends at $j$-th position.
Unless confusion arises, we simply use the notation $\beta$ for $\pi(\beta)$ so that $\beta(i)=j$ whenever $\pi(\beta)(i)=j$.

For $\mathbf{a}\in\left\{ 0,1 \right\}^p$, the set of {\em transposing indices} $T(\mathbf{a})$ of $\mathbf{a}$ consists of all $1\le i\le 2p+2$ such that $|\alpha(\mathbf{a})(i)-i|=1$. For example, $T(0,1,1,0,0,1)=\{3,4,6,9,11,12\}$ and the thicker strands in Figure~\ref{fig:Examplebinarybraid} correspond to transposing indices. For each $i\in T(\mathbf{a})$, the $i$-th strand has no circled crossing in the braid $\alpha(\mathbf{a})$ and so $|T(\mathbf{a})|=p$. In fact, $i\in T(a_1,\ldots,a_p)$ iff for $i=2k$, $a_k=1$ and for $i=2k+1$, $a_k=0$. Thus the function $\alpha:\left\{ 0,1 \right\}^p\to\mathbf{S}_{2p+2}$ is injective.

We define a permutation $\vartheta(\mathbf{a})$ on $(p+2)$ elements 
as corresponding permutation braid is obtained from $\alpha(\mathbf{a})$ by removing strands corresponding to the transposing indices $T(\mathbf{a})$ and renumbering indices as $\{1,2,4,\cdots,2p,2p+1\}$ in order.
For example, in Figure~\ref{fig:Examplebinarybraid},
$\vartheta(\mathbf{a})=(2,8,10,13,12,6,4,1)$
written as a cycle.

The following Lemma states basic properties of $\alpha(\mathbf{a})$ that can be easily derived from the picture of $\alpha(\mathbf{a})$ using building blocks.

\begin{lem}
\label{lem:BasicProperty}
For $\mathbf{a}\in\{0,1\}^p$ with $p\ge 2$, let $\vartheta=\vartheta(\mathbf{a})$.
Then the following holds.
\begin{enumerate}
\item For any $i$, the $i$-th strand cross the strand that ends at the $i$-th puncture in $\alpha(\mathbf{a})$.
\item If $a_i=0$ and $\vartheta(2i)=2j$, then $j$ is the smallest index $j$ such that $i<j$ and $a_j=0$.
\item The permutation $\vartheta$ is always a cycle of length $(p+2)$.
\end{enumerate}
\end{lem}

For positive integers $k$ and $p$,
let $M_{k,p}(\left\{ 0,1 \right\})$ denote the set of all $k\times p$ matrices over $\left\{ 0,1 \right\}$.
For $A\in M_{k,p}(\left\{ 0,1 \right\})$, let $A_i$ and $A^j$ denote the $i$-th row vector and the $j$-th column vector of $A$, respectively.
And let $\mathbf{0}$ and $\mathbf{1}$ denote the column vectors that consist of all 0's and all 1's, respectively.
We extend the function $\alpha$ over $M_{k,p}(\left\{ 0,1 \right\})$ via braid multiplication, that is,
$$\alpha(A)=\alpha(A_1)\cdots\alpha(A_k).$$

For $(A,b)\in M_{k,p}(\{0,1\})\times\{0,\cdots,p\}$, $\alpha(A,b)$ denotes the unique positive $(2p+3)$-braid obtained from $\alpha(A)$ by adding a straight vertical strand between the $(2b+1)$-st puncture and $(2b+2)$-nd puncture of $\alpha(A)$. We may recover $\alpha(A)$ by removing $(2b+2)$-nd strand from $\alpha(A,b)$.

For $n\ge 6$, define
$$
\mathcal{M}_{k,n}=
\begin{cases}
\{A\in M_{k,p}(\{0,1\})|A^1=\mathbf{0}, A^p=\mathbf{1}\},& n=2p+2 \\
\{(A, b)\in M_{k,p}(\{0,1\})\times\{0,\cdots,p\}|
A^1=\mathbf{0}, A^p=\mathbf{1}\},& n=2p+3.
\end{cases}
$$

We now have a function $\alpha:\mathcal{M}_{k,n}\to\B_n$ defined by
$\alpha(A)$ if $n$ is even and $\alpha(A,b)$ if $n$ is odd.
We sometimes use the notations $[A^1\cdots A^p]$ and $[A^1\cdots A^b|A^{b+1}\cdots A^p]$ to denote $\alpha(A)$ and $\alpha(A,b)$, respectively.

\begin{figure}[ht]
\begin{center}
\begin{tabular}{cccc}
\includegraphics[scale=0.85]{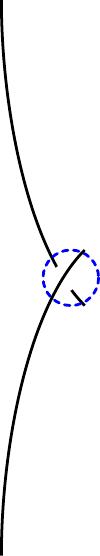}
&
\includegraphics[scale=0.85]{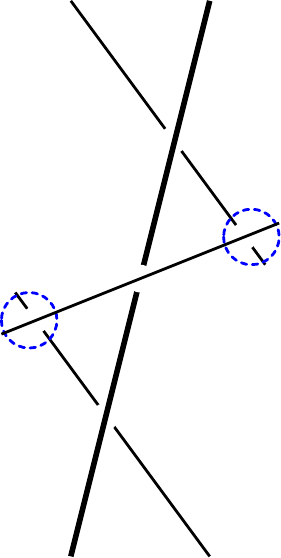}
&
\includegraphics[scale=0.85]{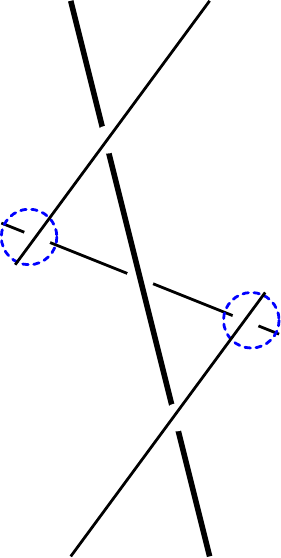}
&
\includegraphics[scale=0.85]{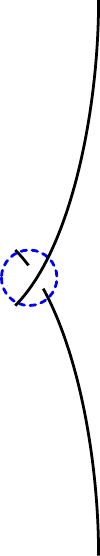}
\end{tabular}
\end{center}
    \caption{Building blocks $\alpha_I, \alpha_0, \alpha_1, \alpha_T$}
	\label{fig:buildingblocks}
    \end{figure}

\begin{figure}[ht]
\begin{center}
\setlength{\unitlength}{334.44574023pt}
\includegraphics[width=\unitlength]{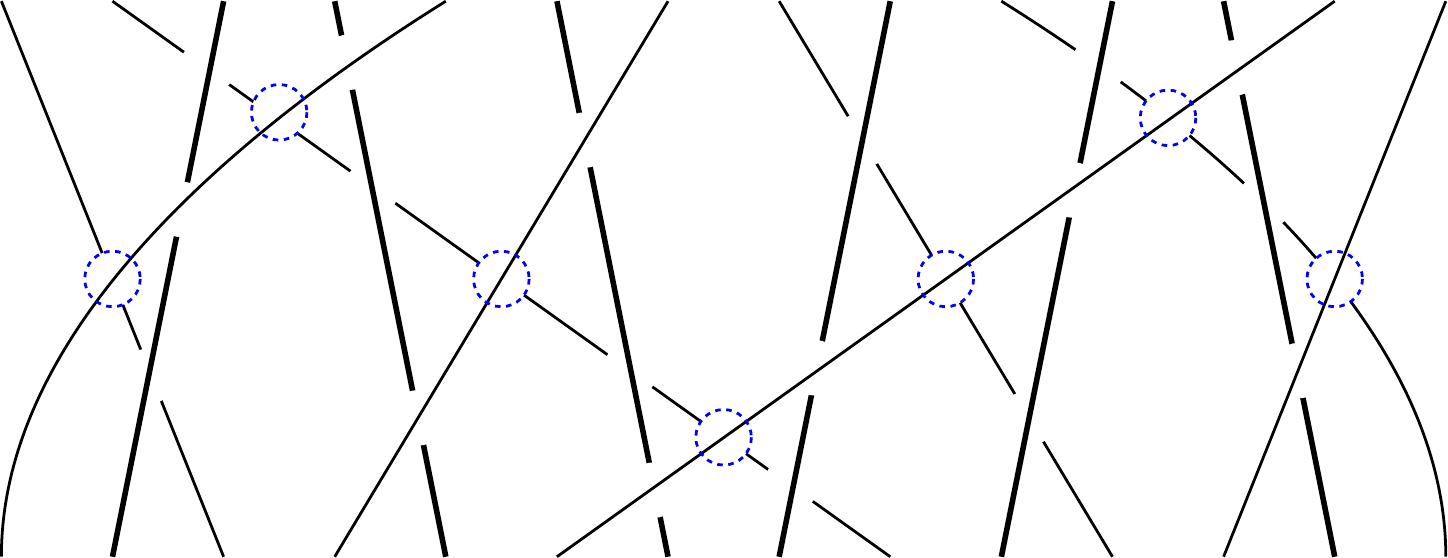}
\end{center}
\caption{The braid $\alpha(0,1,1,0,0,1)$}
\label{fig:Examplebinarybraid}
\end{figure}

\begin{lem}
\label{lem:BinaryBraid}
For $n\ge 6$, the function $\alpha:\mathcal{M}_{k,n}\to\B_n$ has the following properties:
\begin{enumerate}
\item $\alpha(A)=\alpha(A_1)\cdots\alpha(A_k)$ and $\alpha(A,b)=\alpha(A_1,b)\cdots\alpha(A_k,b)$ are left weighted as written and rigid, hence $\alpha$ is injective;
\item There is no standard reduction system for $\alpha(A)$ or $\alpha(A,b)$.
\end{enumerate}
\end{lem}

\begin{proof}
Let $p=\lfloor \frac{n-2}2 \rfloor$.

\noindent (1)
We already observed that $\alpha$ is injective on $\mathcal{M}_{1,n}$.
The starting set $S(\alpha(\mathbf{a}))=\{2,4,\cdots,2p\}$ and the finishing set $F(\alpha(\mathbf{a}))=\{1,2,4,\cdots,2p,2p+1\}$, and so $\alpha(A)=\alpha(A_1)\cdots\alpha(A_k)$ is rigid and this is the left weighted form as written for any $A\in\mathcal{M}_{k,n}$ and so $\alpha$ is injective on $\mathcal{M}_{k,n}$.

Now suppose $n$ is odd. As before, the transposing indices and the vertical strand of $\alpha(\mathbf{a},b)$ completely determine $(\mathbf{a}, b)\in \mathcal{M}_{1,n}$. And the starting and finishing sets of $\alpha(\mathbf{a},b)$ are obtained from those of $\alpha(\mathbf{a})$ by shifting indices $\ge 2b+2$ by 1. Hence $\alpha(A)$ is rigid and left weighted as written and so $\alpha$ is injective on $\mathcal{M}_{k,n}$.

\noindent (2)
By considering punctures contained inside, a standard circle $C$ in a disk with $n$ punctures can be expressed as a subset $\{i,i+1,\ldots,j\}$ of consecutive integers in $\{1,2,\ldots,n\}$ such that $1\le |C|\le n-2$ where $|C|$ denotes the diameter of $C$ defined by $|C|=j-i$.
As a self-homeomorphism of a disk with $n$ punctures, $n$-braids are to act from the right to maintain the consistency between composition and multiplication.
If $C$ is a standard reduction circle for a $n$-braid $\beta$ with the left weighted form $\beta=\beta_1\cdots\beta_k$, then $\beta_1\cdots\beta_i(C)$ is also standard for each $1\le i\le k$ (for example, see \cite{BNG, LL}). In particular, $\beta_1(C)$ is standard and $|\beta_1(C)|=|C|$.

Let $n=2p+2$. Suppose $\alpha(A)$ has a standard circle $C=\{i,i+1,\dots, j\}$ for $A\in\mathcal{M}_{k,n}$. By (1), the left weighted form $\alpha(A)=\beta_1\cdots\beta_k$ is given by $\beta_i=\alpha(A_i)$.
We show that it is impossible to have $|\beta_1(C)|=|C|$ and so $\alpha(A)$ has no standard reduction system.
If $1<i$, then either $\beta_1(i)<i$ or $\beta_1(i+1)\le i$. Similarly, if $j<2p+2$, then either  $\beta_1(j)>j$ or $\beta_1(j-1)\ge j$.
Unless both equalities hold, $|\beta_1(C)|>j-i=|C|$.
Thus $\beta_1(i+1)=i$ and $\beta_1(j-1)=j$. Hence both $i+1$ and $j-1$ are transposing indices and they correspond to 0 and 1 in the sequence $A_1$. Let $j'$ be the non-transposing index corresponding to the last $0$ before the $1$ corresponding to the $(j-1)$-st strand. Then $i\le j'<j-1$ and the $j'$-th strand crosses the $(j-1)$-st strand and $\beta_1(j')>j$. In fact, $\beta_1(j')\ge j+2$ for $j<2p+1$ and $\beta_1(j')=j+1$ for $j=2p+1$.
Hence $|\beta_1(C)|\ge \beta_1(j')-i>j-i=|C|$.

Consider now the case $i=1$. Then $3\le j<2p+2$.
Since $A^1=\mathbf{0}$, $\beta_1(3)=2$.
As above, there exists $j'$ such that $\beta_1(j')\ge j+2$ unless $j=2p+1$.
This implies that $\beta_1(C)\ge (j+2)-2\ge j-1=|C|$ unless $j=2p+1$.
The remaining possibility is that $C=\{1,\cdots, 2p+1\}$ and $\beta_1(2p+2)=1$.
However this is also impossible because $A^p=\mathbf{1}$ and $\beta_1(2p+2)=2p\ge 1$.
The case $j=2p+2$ is symmetric to the case $i=1$.

Let $n$ be odd. Suppose $\alpha(A,b)$ has a standard circle $C=\{i,i+1,\dots, j\}$ for $A\in\mathcal{M}_{k,n}$.
By forgetting the $(2b+2)$-nd strand, consider $C-\{2b+2\}$.
If $|C-\{2b+2\}|>0$, then it is a standard reduction circle for $\alpha(A)$ and so this is impossible.
If $|C-\{2b+2\}|=0$, then $C$ is either $\{2b+1, 2b+2\}$ or $\{2b+2, 2b+3\}$.
But these two cases are impossible by an easy observation on diameters.
\end{proof}

By cycling rows, we define the cycling operation $\mathcal{M}_{k,n}\to \mathcal{M}_{k,n}$ similar to the cycling operation $\mathbf{c}:\B_n\to\B_n$ of braids groups. That is,
$\mathbf{c}(A)$ is obtained from $A$ by moving the first row of $A$ to the bottom so that $\alpha(\mathbf{c}(A))=\alpha(A_2)\cdots\alpha(A_k)\alpha(A_1)$ and $\mathbf{c}(A,b)=(\mathbf{c}(A),b)$.
Note that $\mathbf{c}\alpha=\alpha\mathbf{c}$ by Lemma~\ref{lem:BinaryBraid}~(1).
Hence the order of $A$ or $(A,b)$ under cycling is well defined and denoted by $|\mathbf{c}(A)|$ or $|\mathbf{c}(A,b)|$.

Let $\tau:\B_n\to\B_n$ be the involution sending $\sigma_i$ to $\sigma_{n-i}$. Then for any $\beta\in\B_n$, $\tau(\beta)=\Delta^{-1}\beta\Delta$ where $\Delta$ is the permutation braid corresponding to the half-twist.
For any column vector $\mathbf{a}$ over $\{0,1\}$, let $\mathbf{\bar a}$ denote the column vector obtained by taking Boolean negation for each entry of $\mathbf{a}$.
Define $\tau:\mathcal{M}_{k,n}\to\mathcal{M}_{k,n}$ by
$\tau(A)=(\bar A^p\cdots\bar A^1)$ if $n$ is even and $\tau(A,b)=(\tau(A), p-b)$ if $n$ is odd. It is easy to see that $\tau\alpha=\alpha\tau$.

Given a braid $\beta$, let $I(\beta)$ denote one of finite conjugacy invariant sets $SSS(\beta)$ or $R(\beta)$. For $x\in I(\beta)$, $\rho\in \B_n^+$ is called an {\em $I$-conjugator} of $x$ if $\rho^{-1}x\rho\in I(\beta)$. If a nontrivial $I$-conjugator $\rho\in\B_n^+$ is minimal with respect to the partial order $\prec$, then $\rho$ is called a {\em minimal $I$-conjugator} of $x$. The invariant set $I(\beta)$ is generated by iterated conjugations by minimal $I$-conjugators. Given a left-weighted form $x=\Delta^kx_1\cdots x_{\ell}\in I(\beta)$ and a minimal $I$-conjugator $\rho$ of $x$, there are  two possibilities that $\rho\prec \alpha(x_1)$ or $\rho\prec \alpha(x_\ell)^{-1}\Delta$ and we call $\rho$ a {\em cut-head} or an {\em add-tail} conjugator of $x$, respectively. If $x$ is rigid, the two kinds of minimal $R$-conjugators are mutually exclusive.

For $x,y \in R(\beta)$ and a minimal $R$-conjugator $\rho$,
we use $\xymatrix{x\ar[r]^{\rho}&y}$ if $\rho^{-1}x\rho=y$.
Then $R(\beta)$ can be regarded as a directed graph.
And we use $\xymatrix{x\ar@{=>}[r]^{\rho}&y}$ if there is a directed path from $x$ to $y$ in $R(\beta)$.

\section{Cut-head conjugators for $\alpha(A)$}

For $A\in \mathcal{M}_{k,n}$, we will show that there is only one cut-head $R$-conjugator for $\alpha(A)$.

\begin{lem}\label{lem:StartingSets}
Let $\mathbf{a}=(0,a_2,\cdots,a_{p-1}, 1)\in\{0,1\}^p$ for $p\ge 2$ and $\varphi=\vartheta(\mathbf{a})^{-1}$.
Suppose $\beta\prec\alpha(\mathbf{a})$ and $\beta^{-1}\alpha(\mathbf{a})\gamma$ is a permutation braid for some permutation $(2p+2)$-braid $\gamma$.
\begin{enumerate}
\item[(i)] If $\sigma_{2j}\prec\gamma$ for $1\le j\le p$, then $\sigma_{2j}\sigma_{\varphi(2j)}\prec\beta$;
\item[(ii)] If $\sigma_1\prec\gamma$, then $\sigma_2\sigma_{\varphi(1)}\prec\beta$;
\item[(iii)] If $\sigma_{2p+1}\prec\gamma$, then $\sigma_{2p}\sigma_{\varphi(2p+1)}\prec\beta$.
\end{enumerate}
\end{lem}
\begin{proof}
Consider two strands that end at $2j$-th and $(2j+1)$-st punctures.
One of them corresponds to a transposing index, and the other corresponds to $\varphi(2j)$. If $\sigma_{2j}\prec\gamma$, then $2j\not\in F(\beta^{-1}\alpha(\mathbf{a}))$ since $\beta^{-1}\alpha(\mathbf{a})\gamma$ is a permutation braid. Thus any crossing formed by a strand $s$ and one of the above two strands must be cancelled out by $\beta^{-1}$ if the strand $s$ crosses only one of the two.
In particular, $\sigma_{2j}$ and $\sigma_{\varphi(2j)}$ are such crossings for $2\le j\le p-1$.
Figure~\ref{fig:CutHead} shows $\alpha(\mathbf{a})\sigma_{2j}$ and the shaded region represents $\beta\prec\alpha(\mathbf{a})$ and two crossings surrounded by dotted circles are $\sigma_{2j}$ and $\sigma_{\varphi(2j)}$.
For $j=1$, $\varphi(2)=1$ and it is easy to see that $\sigma_2\sigma_1\prec\beta$ to have $2\not\in F(\beta^{-1}\alpha(\mathbf{a}))$. Similarly we have $\sigma_{2p}\sigma_{2p+1}\prec\beta$ if $2p\not\in F(\beta^{-1}\alpha(\mathbf{a}))$. Thus we establishes (i).
The cases (ii) and (iii) are similar to the case (i) with $2\le j\le p-1$.
\end{proof}

\begin{figure}[ht]
\begin{center}
\begingroup
    \setlength{\unitlength}{204pt}
  \begin{picture}(1,0.73637727)%
    \put(0,0){\includegraphics[width=\unitlength]{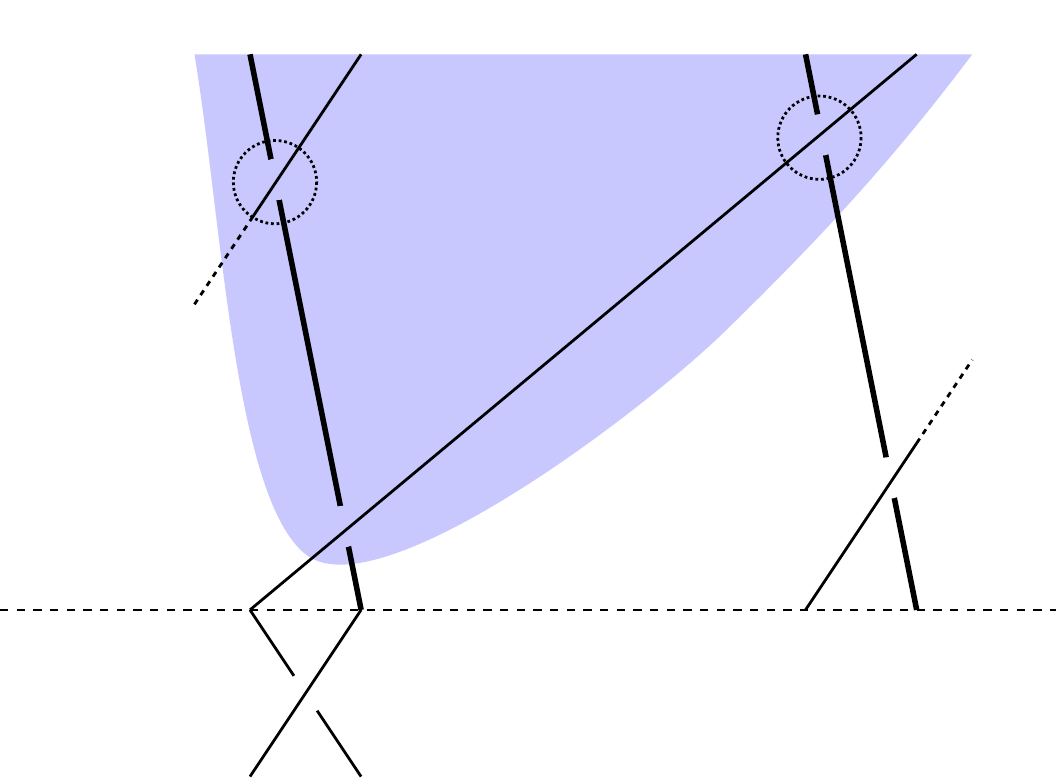}}%
    \put(0.18421053,0.16798485){\color[rgb]{0,0,0}\makebox(0,0)[lb]{\smash{$2j$}}}%
    \put(0.38345867,0.16798485){\color[rgb]{0,0,0}\makebox(0,0)[lb]{\smash{$2j+1$}}}%
    \put(0.72368421,0.69806005){\color[rgb]{0,0,0}\makebox(0,0)[lb]{\smash{$\varphi(2j)$}}}%
    \put(0.06954887,0.36347359){\color[rgb]{0,0,0}\makebox(0,0)[lb]{\smash{$\alpha(\mathbf{a})$}}}%
    \put(0.07706767,0.04016531){\color[rgb]{0,0,0}\makebox(0,0)[lb]{\smash{$\sigma_{2j}$}}}%
    \put(0.41729323,0.48941341){\color[rgb]{0,0,0}\makebox(0,0)[lb]{\smash{$\beta$}}}%
    \put(0.22021053,0.69806005){\color[rgb]{0,0,0}\makebox(0,0)[lb]{\smash{$2j$}}}%
  \end{picture}%
\endgroup
\end{center}
\caption{An intuitive figure for the proof of Lemma~\ref{lem:StartingSets}}
\label{fig:CutHead}
\end{figure}
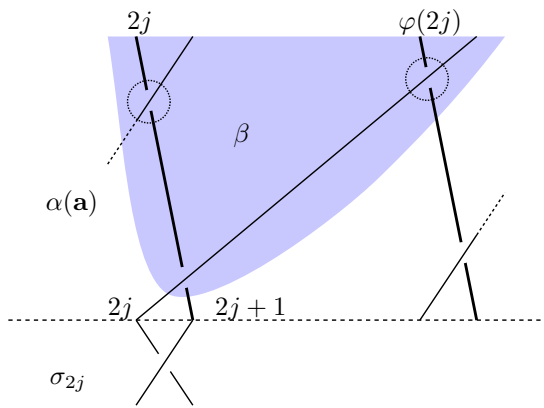

\begin{lem}
\label{lem:CutheadMinimalConjugator}
For $A\in \mathcal{M}_{k,n}$,
there is only one cut-head minimal $R$-conjugator $\alpha(A_1)$ of $\alpha(A)$.
\end{lem}
\begin{proof}
Let $\rho$ be an arbitrary cut-head $SSS$-conjugator for $\alpha(A)$. Suppose $n=2p+2$.
Since $S(\alpha(A_1))=\{2,4,\cdots,2p\}$ and $\rho$ is nontrivial, $2j\in S(\rho)$ for some $1\le j\le p$.
Also there exists $\rho_i$'s that make
$$
\rho^{-1}\alpha(A)\rho=(\rho^{-1}\alpha(A_1)\rho_1)\cdots(\rho_{k-1}^{-1}\alpha(A_k)\rho)
$$
left-weighted as written. Each block in this decomposition satisfies the hypothesis of Lemma~\ref{lem:StartingSets}. Since $\rho$ feeds back at the end and $\varphi$ is a cycle on the set $\{1,2,\ldots,2p,2p+1\}$ by Lemma~\ref{lem:BasicProperty}, iterated application of Lemma~\ref{lem:StartingSets} shows  $$(\sigma_2\sigma_1)\sigma_4\cdots\sigma_{2p-2}(\sigma_{2p}\sigma_{2p+1})$$ is a prefix of $\rho$ and $\rho_i$ for all $i$. We note that $\sigma_2\sigma_1$, $\sigma_4,\ldots,\sigma_{2p-2}$, $\sigma_{2p}\sigma_{2p+1}$ are mutually commutative and so they are not interfered each other. In particular, $(\sigma_2\sigma_1)\sigma_4\cdots\sigma_{2p-2}(\sigma_{2p}\sigma_{2p+1})\prec\rho_1$ implies $F(\rho^{-1}\alpha(A_1))\cap\{1,2,\ldots,2p,2p+1\}=\emptyset$. Since $F(\rho^{-1}\alpha(A_1))\subset\{1,2,\ldots,2p,2p+1\}$, we must have $F(\rho^{-1}\alpha(A_1))=\emptyset$, that is, $\rho=\alpha(A_1)$. Since $\alpha(A_1)$ is clearly an $R$-conjugator of $\alpha(A)$, it is a unique cut-head $R$-conjugator of $\alpha(A)$.

Now suppose $n=2p+3$. Let $(A,b)\in\mathcal{M}_{k,n}$, $\rho$ be an arbitrary cut-head $SSS$-conjugator for $\alpha(A,b)$, and $\rho'$ be obtained from $\rho$ by removing $(2b+2)$-nd strand. Since the starting and finishing sets of $\alpha(A_i,b)$ are obtained from those of $\alpha(A_i)$ by shifting indices $\ge 2b+2$ by 1 for all $i$, $\rho'$ is a cut-head $SSS$-conjugator of $\alpha(A)$. Thus $\rho'=\alpha(A_1)$ and so
$\rho=\alpha(A_1,b)$. Then $\alpha(A_1,b)$ is clearly an $R$-conjugator of $\alpha(A,b)$ and so it is a unique cut-head $R$-conjugator of $\alpha(A,b)$.
\end{proof}

\section{Add-tail conjugators for $\alpha(A)$}
Suppose $n$ is even.
For $A\in\mathcal{M}_{k,n}$, let $\rho$ be an arbitrary add-tail minimal $R$-conjugator for $\alpha(A)$.
Then $F(\alpha(A_k))=\{1,2,4,\cdots,2p,2p+1\}$ and so $S(\rho)\subset\{3,5,\cdots,2p-1\}$. Hence we only consider add-tail conjugator starting with $\sigma_{2i+1}$ for $1\le i\le p-1$.
\begin{lem}
\label{lem:Join}
Let $\mathbf{a}\in\{0,1\}^p$ for $p\ge 2$, and $\vartheta=\vartheta(\mathbf{a})$.
Suppose $\alpha(\mathbf a)\gamma$ is a permutation braid for some permutation $(2p+2)$-braid $\gamma$
and  $\sigma_{2i+1}\prec\alpha(\mathbf a)\gamma$ for $1\le i\le p-1$.
\begin{enumerate}
\item[(i)] If $a_i=0$ and $a_{i+1}=1$, then $\sigma_{2i+1}\sigma_{2i}\sigma_{2i+2}\sigma_{2i+1}\prec\gamma$;
\item[(ii)] If $a_{i+1}=0$, then
$\sigma_{2i+1}(\sigma_{2i+3}\sigma_{2i+4}\cdots\sigma_{\vartheta(2i+2)})\prec\gamma$; 
\item[(iii)] If $a_{i+1}=0$ and $\sigma_{2i+3}\sigma_{2i+4}\cdots\sigma_{2j-1}\prec\alpha(\mathbf a)\gamma$ for the smallest $j>i+1$ such that $a_j=0$, then $\sigma_{2j+1}\prec\gamma$;
\item[(iv)] If $a_i=1$, then $\sigma_{2i+1}(\sigma_{2i-1}\sigma_{2i-2}\cdots\sigma_{\vartheta(2i)})\prec\gamma$; 
\item[(v)] If $a_i=1$ and $\sigma_{2i-1}\sigma_{2i-2}\cdots\sigma_{2j+1}\prec\alpha(\mathbf a)\gamma$ for the largest $j<i$ such that $a_j=1$, then $\sigma_{2j-1}\prec\gamma$.
\end{enumerate}
\end{lem}
\begin{proof}
\noindent (i) This follows from Figure~\ref{fig:AddTail}(a). The crossing $a=\sigma_{2i+1}^{-1}$ in the picture must be canceled by multiplying $\gamma$ and so $\gamma$ has the crossing $A$. Since strands of $\alpha(\mathbf a)\gamma$ can cross each other at most once, crossings $b$ and $c$ force crossings $B$ and $C$, respectively.

Recall the permutation $\vartheta=\vartheta(\mathbf{a})$ of non-transposing indices of $\mathbf a$. If $a_{i+1}=0$, $\vartheta(2i+2)=2j$ for the smallest $j>i+1$ such that $a_j=0$ if such $j$ exists or $\vartheta(2i+2)=2p+1$ otherwise by Lemma~\ref{lem:BasicProperty}.

(ii) In Figure~\ref{fig:AddTail}(b), the shaded band represents a bundle of strands. The crossing $a=\sigma_{2i+1}^{-1}$ must be canceled by multiplying $\gamma$ and so $\gamma$ has the crossing $A$.  Since $\alpha(\mathbf a)\gamma$ is a permutation braid, the crossings $b$, $c$, $d$, and $e$ in $\alpha(\mathbf a)$ force the corresponding crossings $B$, $C$, $D$, and $E$ in $\gamma$, respectively. In particular, $ECB\prec\gamma$ where $C=\sigma_{2i+3}\dots\sigma_{\vartheta(2i+2)-1}$, $B=\sigma_{\vartheta(2i+2)}$, and $E=\sigma_{2i+1}$.

(iii) In Figure~\ref{fig:AddTail}(c), the crossings $(\sigma_{2i+3}\sigma_{2i+4}\cdots\sigma_{2j})^{-1}$ must be canceled by $\gamma$. In particular. $\gamma$ must have the crossing $A$ canceling $a$. Since the crossing $b$ forces the crossing $B$ and $\vartheta(2i+2)=2j$, $\sigma_{2j+1}\dots\sigma_{\vartheta(2j)}\prec\gamma$ and so $\sigma_{2j+1}\prec\gamma$.

(iv) and (v) are equivalent to (ii) and (iii) via the involution $\tau$.
\end{proof}

\begin{figure}[ht]
\begin{center}
\subfigure[{$[\cdots 01\cdots]$}]{
    \setlength{\unitlength}{80.8pt}
  \begin{picture}(1,2.32673267)%
    \put(0,0){\includegraphics[width=\unitlength]{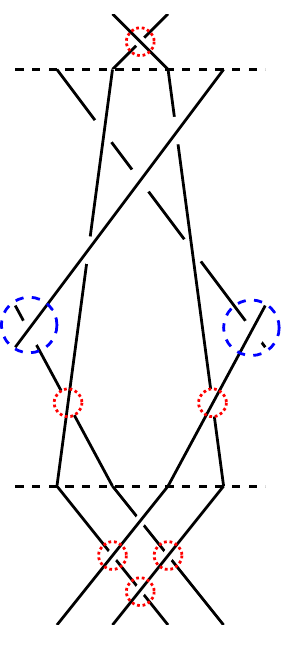}}%
    \put(0.56,2.15){\color[rgb]{0,0,0}\makebox(0,0)[lb]{\smash{$a$}}}%
    \put(0.10,0.86){\color[rgb]{0,0,0}\makebox(0,0)[lb]{\smash{$b$}}}%
    \put(0.83,0.86){\color[rgb]{0,0,0}\makebox(0,0)[lb]{\smash{$c$}}}%
    \put(0.56,0.16){\color[rgb]{0,0,0}\makebox(0,0)[lb]{\smash{$A$}}}%
    \put(0.66,0.31){\color[rgb]{0,0,0}\makebox(0,0)[lb]{\smash{$B$}}}%
    \put(0.22,0.31){\color[rgb]{0,0,0}\makebox(0,0)[lb]{\smash{$C$}}}%
  \end{picture}
}
\subfigure[{$[\cdots *0\cdots]$}]{
    \setlength{\unitlength}{112pt}
  \begin{picture}(1,1.67857143)%
    \put(0,0){\includegraphics[width=\unitlength]{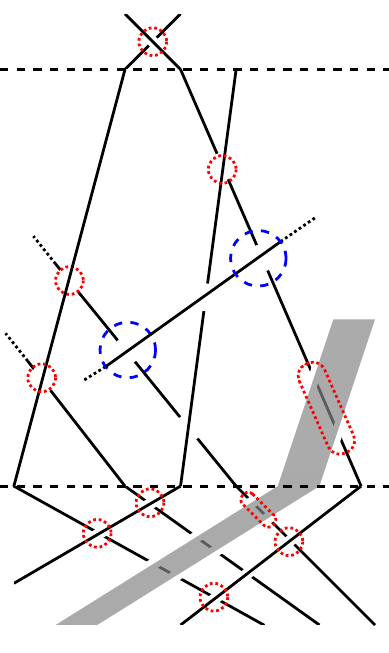}}%
    \put(0.44,1.55){\color[rgb]{0,0,0}\makebox(0,0)[lb]{\smash{$a$}}}%
    \put(0.24,0.92){\color[rgb]{0,0,0}\makebox(0,0)[lb]{\smash{$b$}}}%
    \put(0.70,0.64){\color[rgb]{0,0,0}\makebox(0,0)[lb]{\smash{$c$}}}%
    \put(0.63,1.20){\color[rgb]{0,0,0}\makebox(0,0)[lb]{\smash{$d$}}}%
    \put(0.16,0.68){\color[rgb]{0,0,0}\makebox(0,0)[lb]{\smash{$e$}}}%
   \put(0.4,0.10){\color[rgb]{0,0,0}\makebox(0,0)[lb]{\smash{$A$}}}%
    \put(0.81,0.25){\color[rgb]{0,0,0}\makebox(0,0)[lb]{\smash{$B$}}}%
    \put(0.51,0.34){\color[rgb]{0,0,0}\makebox(0,0)[lb]{\smash{$C$}}}%
    \put(0.08,0.26){\color[rgb]{0,0,0}\makebox(0,0)[lb]{\smash{$D$}}}%
    \put(0.33,0.46){\color[rgb]{0,0,0}\makebox(0,0)[lb]{\smash{$E$}}}%
  \end{picture}
}
\subfigure[{$[\cdots01\cdots10\cdots]$}]{
  \setlength{\unitlength}{204.4pt}
  \begin{picture}(1,0.91976517)%
    \put(0,0){\includegraphics[width=\unitlength]{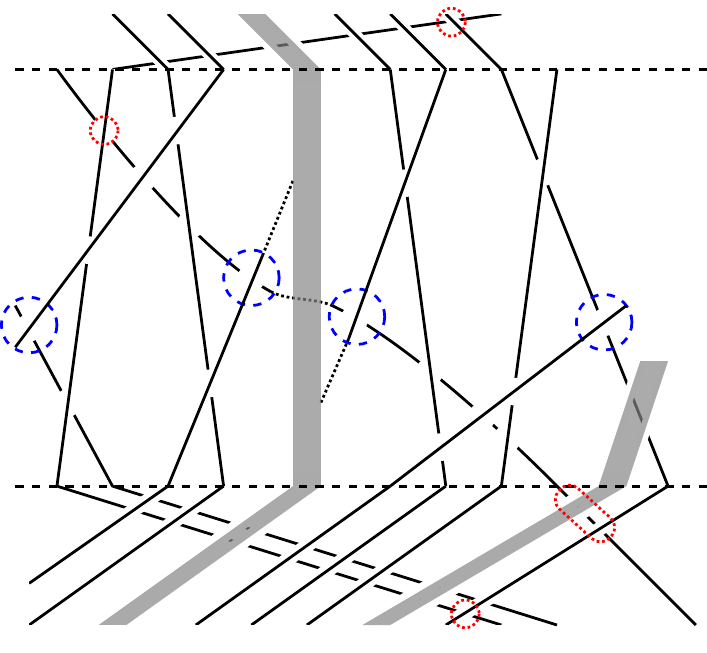}}%
    \put(0.63,0.84){\color[rgb]{0,0,0}\makebox(0,0)[lb]{\smash{$a$}}}%
    \put(0.09,0.72){\color[rgb]{0,0,0}\makebox(0,0)[lb]{\smash{$b$}}}%
    \put(0.59,0.03){\color[rgb]{0,0,0}\makebox(0,0)[lb]{\smash{$A$}}}%
    \put(0.89,0.16){\color[rgb]{0,0,0}\makebox(0,0)[lb]{\smash{$B$}}}%
    \put(0.00,0.84){\color[rgb]{0,0,0}\makebox(0,0)[lb]{\smash{$2i+2$}}}%
    \put(0.78,0.25){\color[rgb]{0,0,0}\makebox(0,0)[lb]{\smash{$2j$}}}%
  \end{picture}
}
\end{center}
\caption{Forced accumulation by $\sigma_{2i+1}$}
\label{fig:AddTail}
\end{figure}

In addition we need to consider the following possibilities for odd braids.

\begin{lem}\label{lem:JoinOdd}
Let $\mathbf{a}, \vartheta$ be as before and $1\le b\le p-1$ be an integer. Suppose $\alpha(\mathbf{a},b)\gamma$ is a permutation braid for some permutation $(2p+3)$-braid $\gamma$ and $\sigma_{2b+e}\prec\alpha(\mathbf{a},b)\gamma$ for some $1\le e\le 2$.
\begin{enumerate}
\item [(i)] If $e=1$ and $a_b=0$, then $\sigma_{2b+1}\sigma_{2b}\prec\gamma$;
\item [(ii)] If $e=1$ and $a_b=1$, then $\sigma_{2b+1}(\sigma_{2b-1}\dots\sigma_{\vartheta(2i)})\prec\gamma$;
\item [(iii)] If $e=2$ and $a_{b+1}=0$, then $\sigma_{2b+2}(\sigma_{2b+4}\dots\sigma_{\vartheta(2i+2)})\prec\gamma$.
\item [(iv)] If $e=2$ and $a_{b+1}=1$, then $\sigma_{2b+2}\sigma_{2b+3}\prec\gamma$;
\end{enumerate}
\end{lem}
\begin{proof}
By using Figure~\ref{fig:AddTailOdd} and the argument similar to the proof of the previous lemma, (i) and (ii) follow. The assertions (iii) and (iv) also follow using the involution $\tau$.
\end{proof}

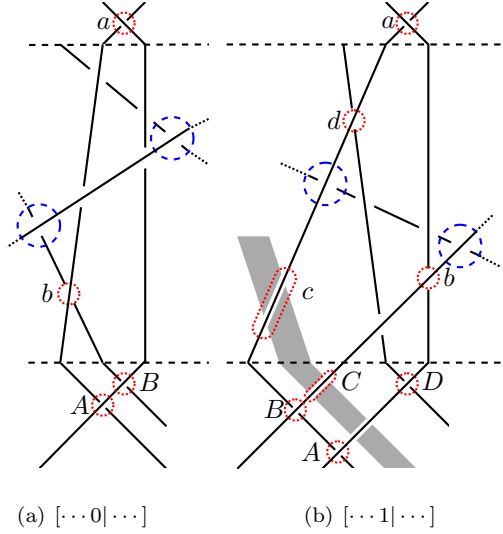
\begin{figure}[ht]
\begin{center}
\subfigure[{$[\cdots0|\cdots]$}]{
    \setlength{\unitlength}{92.05bp}%
  \begin{picture}(1,2.04236785)%
    \put(0,0){\includegraphics[width=\unitlength]{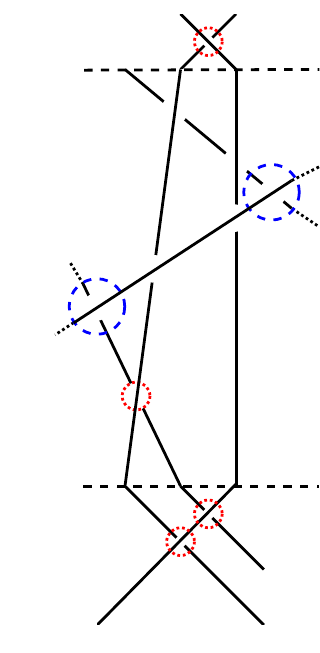}}%
    \put(0.54042629,1.88665997){\color[rgb]{0,0,0}\makebox(0,0)[lb]{\smash{$a$}}}%
    \put(0.31472644,0.77076783){\color[rgb]{0,0,0}\makebox(0,0)[lb]{\smash{$b$}}}%
    \put(0.71239147,0.40426037){\color[rgb]{0,0,0}\makebox(0,0)[lb]{\smash{$B$}}}%
    \put(0.43,0.31){\color[rgb]{0,0,0}\makebox(0,0)[lb]{\smash{$A$}}}%
  \end{picture}%
}
\subfigure[{$[\cdots1|\cdots]$}]{
    \setlength{\unitlength}{108.00080283bp}%
  \begin{picture}(1,1.7407278)%
    \put(0,0){\includegraphics[width=\unitlength]{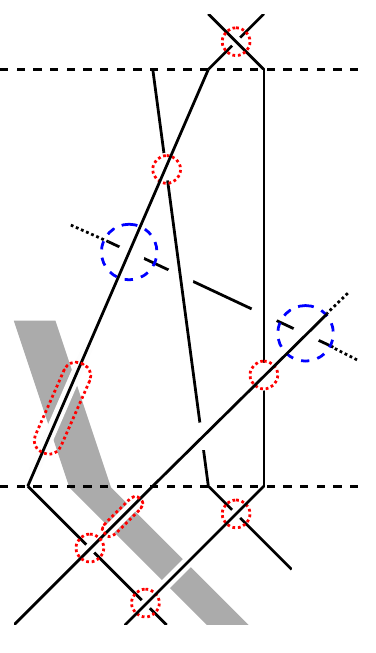}}%
    \put(0.54039712,1.61175765){\makebox(0,0)[lb]{\smash{$a$}}}%
    \put(0.75702588,0.71118508){\makebox(0,0)[lb]{\smash{$b$}}}%
    \put(0.13,0.25){\makebox(0,0)[lb]{\smash{$B$}}}%
    \put(0.26197701,0.66533802){\makebox(0,0)[lb]{\smash{$c$}}}%
    \put(0.40,0.35){\makebox(0,0)[lb]{\smash{$C$}}}%
    \put(0.26,0.10){\makebox(0,0)[lb]{\smash{$A$}}}%
    \put(0.35043832,1.26008174){\makebox(0,0)[lb]{\smash{$d$}}}%
    \put(0.67992119,0.35){\makebox(0,0)[lb]{\smash{$D$}}}%
  \end{picture}
}
\end{center}
\caption{Forced accumulation by $\sigma_{2b+1}$}
\label{fig:AddTailOdd}
\end{figure}

Lemmas~\ref{lem:Join} and \ref{lem:JoinOdd} apply to each row of $A\in\mathcal{M}_{k,n}$ and every contribution forced by the lemma must accumulate on an add-tail $SSS$-conjugator $\rho$ of $\alpha(A)$ to avoid an increase of the canonical length of $\rho^{-1}\alpha(A)\rho$. We in general expect to have a long add-tail $SSS$-conjugator for $\alpha(A)$ and the randomness among rows $A$ defeats our purpose since longer add-tail minimal $R$-conjugators of $\alpha(A)$ generate a smaller rigid conjugacy set $R(\alpha(A))$. In order to make our discussion the simplest possible, we consider the following subset $\mathcal{M}_{k,n}^0$ of $\mathcal{M}_{k,n}$ that still allows enough add-tail conjugators to serve our purpose.
For $n\ge 10$ and $k\ge 2$, we define
$$
\mathcal{M}_{k,n}^0=\{A\:|\: A^i=\mathbf0 \text{ or }\mathbf1 \text{ for } 2\le i\le p-3,\text{ and } |\mathbf{c}(A)|=k\} \quad \text{ for $n=2p+2$}
$$
and
$$
\mathcal{M}_{k,n}^0=\{(A,p-2)\:|\: A\in\mathcal{M}_{k,n-1}^0 \} \quad \text{ for $n=2p+3$}.
$$
Since the location of the vertical strand is fixed at $p-2$ for an odd n, we simply denote an element of $\mathcal{M}_{k,n}^0$ by a matrix $A\in M_{k,p}(\{0,1\})$ if no ambiguity occurs. Notice that a $k\times p$ binary matrix $A$ in $\mathcal{M}_{k,n}^0$ has constant columns except the $(p-2)$-nd and the $(p-1)$-st columns. We assume $n\ge 10$ and $k\ge 2$ whenever we consider $\mathcal{M}_{k,n}^0$ since the notation makes no sense otherwise.

\begin{lem}
\label{lem:AddtailMinimalConjugator}
Let $A\in \mathcal{M}_{k,n}^0$ and $p=\lfloor\frac{n-2}2\rfloor$. If there is $2\le j\le p-3$ such that $A^j=\mathbf 1$, then every add-tail minimal $R$-conjugator for $\alpha(A)$ is either of the form $\rho_{2i+1}=\sigma_{2i+1}\sigma_{2i}\sigma_{2i+2}\sigma_{2i+1}$ for $2\le i\le p-4$ or of the form $\rho_3=\sigma_3\sigma_2\sigma_4\sigma_3\sigma_1$. Moreover, $\rho_{2i+1}$ arises exactly when $A^i=\mathbf{0}$ and $A^{i+1}=\mathbf{1}$.
\end{lem}
\begin{proof}
First consider even braids, that is, $n=2p+2$.
If $A^i=\mathbf{0}$ and $A^{i+1}=\mathbf{1}$ for $2\le i\le p-4$, then $\rho_{2i+1}$ is clearly an add-tail minimal $R$-conjugator in the view of Lemma~\ref{lem:Join}(i) and in fact
$$
\xymatrix{
[\cdots\mathbf{01}\cdots]\ar[r]^-{\rho_{2i+1}}&[\cdots\mathbf{10}\cdots]
}.
$$
If $A^1=\mathbf{0}$ and $A^{2}=\mathbf{1}$, then $\rho_{3}$ is also an add-tail minimal $R$-conjugator in the view of Lemma~\ref{lem:Join}(i) except the fact that the extra $\sigma_1$ is needed to make the result rigid and in fact
$$
\xymatrix{[\mathbf{01}\cdots]\ar[r]^-{\rho_3}&[\mathbf{00}\cdots]}.
$$

On the other hand, let $\rho$ be an arbitrary add-tail minimal $R$-conjugator. Since $F(\alpha(A))=\{1,2,4,\ldots,2p,2p+1\}$, we may assume that $\sigma_{2i+1}\prec\rho$ some $1\le i\le p-1$. It is suffice to show that this is impossible if $A^i\ne\mathbf{0}$ or $A^{i+1}\ne\mathbf{1}$.

Suppose $A^i=\mathbf0$ and $A^{i+1}\ne\mathbf{1}$. Then Lemma~\ref{lem:Join}(ii) and (iii) are applicable and $2j+1\in S(\rho)$ for the smallest $j>i+1$ such that $A^j\ne\mathbf 1$. If $A^j=\mathbf 0$ and $A^{j+1}\ne\mathbf{1}$, we go back and repeat the argument by replacing $i$ by $j$. Since $A^{p-2}\ne \mathbf 1$, this repetition eventually yields the cases considered next. If $A^j=\mathbf 0$ and $A^{j+1}=\mathbf{1}$, an add-tail minimal $R$-conjugator $\rho_{2j+1}$ starting with $2j+1$ exists and it does not start with $\sigma_{2i+1}$ and this contradicts the minimality of $\rho$. If $A^j\ne\mathbf 0$, then $j=p-2$ and this is the case we consider next since $A^{p-2}\ne\mathbf0$.

Suppose $A^i\ne\mathbf0$. Then Lemma~\ref{lem:Join}(iv) and (v) are applicable and so $2(j-1)+1\in S(\rho)$ for the largest $j<i$ such that $A^{j}=\mathbf{1}$. If $A^{j-1}\ne\mathbf{0}$, we we go back and repeat the argument by replacing $i$ by $j-1$. Since $A^1=\mathbf 0$, this repetition eventually yields the case considered next. If $A^{j-1}=\mathbf{0}$, an add-tail minimal $R$-conjugator $\rho_{2j-1}$ starting with $2j-1$ exists and it does not start with $\sigma_{2i+1}$ and this contradicts the minimality of $\rho$.

Now consider odd braids, that is, $n=2p+3$. When $A^i=\mathbf{0}$ and $A^{i+1}=\mathbf{1}$ for $i=1,\ldots,p-4$, the existence of the add-tail minimal $R$-conjugators $\rho_{2i+1}$ is identical to the case of even braids. Let $\rho$ be an add-tail minimal $R$-conjugator for $\alpha(A)$. Then $S(\rho)\subset\{3,5,\dots, 2p-3, 2p-2, 2p\}$ since
$$F(\alpha(A))=\{1,2,4,\dots,2p-4, 2p-1, 2p+1, 2p+2\}.$$
Let $\rho'$ be a braid obtained by removing the $(2p-2)$-nd strand from $\rho$.
Then $\rho'$ is an add-tail $SSS$-conjugator for $\alpha(A)\in \B_{n-1}$.

If $\sigma_{2i+1}\prec\rho$ for some $1\le i\le p-3$ such that $A^i\ne\mathbf{0}$ or $A^{i+1}\ne\mathbf{1}$ or $\sigma_{2p}\rho$, then $\sigma_{2i+1}\prec\rho$ for some $1\le i\le p-1$ such that $A^i\ne\mathbf{0}$ or $A^{i+1}\ne\mathbf{1}$. By the above argument for even braids, $2j+1\in S(\rho')$ for some $1\le j \le p-4$ such that $A^j=\mathbf 0$ and $A^{j+1}=\mathbf{1}$. By inserting the vertical strand back to $\rho'$, we have $2j+1\in S(\rho)$ for some $1\le j \le p-4$ such that $A^j=\mathbf 0$ and $A^{j+1}=\mathbf{1}$. Since the minimal $R$-conjugator $\rho_{2j+1}$ exits and does not start with $\sigma_{2i+1}$, $\rho$ cannot be minimal.

Finally if $2p-3$ or $2p-2$ are in $S(\rho)$,
then by Lemma~\ref{lem:JoinOdd}, $\sigma_{2i+1}\prec\rho$ for some $1\le i\le p-3$ or $\sigma_{2p}\prec\rho$ since both $A^{p-2}$ and $A^{p-1}$ are neither $\mathbf{0}$ nor $\mathbf{1}$. We have already shown that $\sigma_{2j+1}\prec\rho$ for some $1\le j \le p-4$ such that $A^j=\mathbf 0$ and $A^{j+1}=\mathbf{1}$ under this situation. Since $\rho_{2j+1}$ starts with neither $\sigma_{2p-3}$ nor $\sigma_{2p-2}$, $\rho$ cannot be minimal.
\end{proof}

Lemma~\ref{lem:AddtailMinimalConjugator} completely determined all add-tail minimal $R$-conjugators of $\alpha(A)$ for $A\in\mathcal{M}_{k,n}^0$ except the case when $A^1=A^2=\cdots=A^{p-3}=\mathbf 0$. In fact we will show that such a braid has a unique add-tail minimal $R$-conjugator. However it seems difficult to find the minimal $R$-conjugator via a direct computation since it is quite long.

As we observed in the proof of Lemma~\ref{lem:AddtailMinimalConjugator}, the conjugation by a minimal $R$-conjugator $\rho_{2i+1}$ for $2\le i\le p-4$ exchanges the $i$-th column of 0's and the $(i+1)$-st column of 1's in $A\in \mathcal{M}_{k,n}^0$ and the conjugation by $\rho_{3}$ turns the 2nd column of 1's into the column of 0's. Thus we have the following diagrams:
$$
\xymatrix{
[\cdots\mathbf{01}\cdots]\ar[r]^-{\rho_{2i+1}}&[\cdots\mathbf{10}\cdots],
\ [\mathbf{01}\cdots]\ar[r]^-{\rho_3}&[\mathbf{00}\cdots]
}
$$
We call these conjugations {\em switching}.
It is clear that for $i=1,p-2,p-1,p$, the $i$-th column and the vertical strand of an odd braid remains unchanged by switchings and so switchings are operations within
$\mathcal{M}_{k,n}^0$.

For $A\in\mathcal{M}_{k,n}^0$,
we may regard the rigid conjugacy set $R(\alpha(A))$ as a directed graph with each edge labeled by a minimal $R$-conjugator. Let $\Gamma(A)$ be the directed graph generated by switchings and their inverses starting from $\alpha(A)$. Then $\Gamma(A)$ forms a connected subgraph of $R(\alpha(A))$. Let
$$\hat A=(\mathbf{0}\mathbf{1}\cdots\mathbf{1}A^{p-2}A^{p-1}\mathbf{1})
\quad\text{and}\quad\check{A}=(\mathbf{0}\mathbf{0}\cdots\mathbf{0}A^{p-2}A^{p-1}\mathbf{1}).$$
Given any $A\in\mathcal{M}_{k,n}^0$, it is clear that we can convert $\hat A$  into $A$ and $A$ into $\check A$ via series of switchings. Thus the directed graph $\Gamma(A)$ forms a lattice and $\alpha(\hat A)$ and $\alpha(\check A)$ are the unique initial element and the unique terminal element, respectively.

If there are two directed paths from $\beta_0$ to $\beta_1$ in $\Gamma(A)$, that is, $\gamma^{-1}\beta_0\gamma=\beta_1=\gamma'^{-1}\beta_0\gamma'$ for some $\gamma$, $\gamma'$ that are products of add-tail minimal $R$-conjugators, then the region between the two directed paths is completely filled by squares of the following type:
$$
\xymatrixcolsep{3pc}
\xymatrix{
[\cdots\mathbf{01}\cdots\mathbf{01}\cdots]\ar@{}[rd]|\circlearrowleft\ar[r]^{\rho_{2j+1}}\ar[d]_{\rho_{2j'+1}}&[\cdots\mathbf{10}\cdots\mathbf{01}\cdots]\ar[d]^{\rho_{2j'+1}}\\
[\cdots\mathbf{01}\cdots\mathbf{10}\cdots]\ar[r]^{\rho_{2j+1}}&[\cdots\mathbf{10}\cdots\mathbf{10}\cdots]
}
$$
where $|j-j'|\ge 2$ by the requirement on $\beta_0$ for $\rho_{2j+1}$ and $\rho_{2j'+1}$ to exist as described in Lemma~\ref{lem:AddtailMinimalConjugator}.
Since $\rho_{2j+1}\rho_{2j'+1}=\rho_{2j'+1}\rho_{2j+1}$, we have $\gamma=\gamma'$.
Let $\rho_A$ denote the unique braid that is a product of add-tail minimal $R$-conjugators such that $\rho_A^{-1}\alpha(\hat A)\rho_A=\alpha(A)$.

\begin{exam}
\label{exm:Exam}
Let $\alpha(A)=\left[\genfrac{}{}{0pt}{}{0\ 1\ 0\ 1\ \mathit{1}}{0\ 1\ 0\ 1\ \mathit{0}}\big\vert
\genfrac{}{}{0pt}{}{\mathit{0}\ 1}{\mathit{1}\ 1}\right]\in \B_{17}$.
Two columns that are not supposed to be constant are indicated by slanted entries.

Then $\alpha(\hat A)=
\left[\genfrac{}{}{0pt}{}{0\ 1\ 1\ 1\ \mathit{1}}{0\ 1\ 1\ 1\ \mathit{0}}\big\vert
\genfrac{}{}{0pt}{}{\mathit{0}\ 1}{\mathit{1}\ 1}\right]$ and
$\alpha(\check A)=
\left[\genfrac{}{}{0pt}{}{0\ 0\ 0\ 0\ \mathit{1}}{0\ 0\ 0\ 0\ \mathit{0}}\big\vert
\genfrac{}{}{0pt}{}{\mathit{0}\ 1}{\mathit{1}\ 1}\right]$.
And the following diagram shows $\alpha(\hat A), \alpha(A)$ and $\alpha(\check A)$ in the direct graph $\Gamma(A)$.

$$
\xymatrixcolsep{4pc}
\xymatrix{
\alpha(\hat A)=\left[\genfrac{}{}{0pt}{}{0\ 1\ 1\ 1\ \mathit{1}}{0\ 1\ 1\ 1\ \mathit{0}}\big\vert
\genfrac{}{}{0pt}{}{\mathit{0}\ 1}{\mathit{1}\ 1}\right]
\ar[r]^-{\rho_3}
\ar@{=>}[dr]^{\rho_{A}}
\ar@{=>}[d]^{\rho_{\check A}}
&
\left[\genfrac{}{}{0pt}{}{0\ 0\ 1\ 1\ \mathit{1}}{0\ 0\ 1\ 1\ \mathit{0}}\big\vert
\genfrac{}{}{0pt}{}{\mathit{0}\ 1}{\mathit{1}\ 1}\right]\hphantom{=\alpha(A)}
\ar@<-3ex>[d]^{\rho_{5}}\\
\alpha(\check A)=\left[\genfrac{}{}{0pt}{}{0\ 0\ 0\ 0\ \mathit{1}}{0\ 0\ 0\ 0\ \mathit{0}}\big\vert
\genfrac{}{}{0pt}{}{\mathit{0}\ 1}{\mathit{1}\ 1}\right]
&
\left[\genfrac{}{}{0pt}{}{0\ 1\ 0\ 1\ \mathit{1}}{0\ 1\ 0\ 1\ \mathit{0}}\big\vert
\genfrac{}{}{0pt}{}{\mathit{0}\ 1}{\mathit{1}\ 1}\right]=\alpha(A)
\ar[l]_-{\rho_3\rho_7\rho_5\rho_3}
}
$$

And we have
\begin{align*}
\rho_A&=\rho_3\rho_5\\
&=(\sigma_3\sigma_2\sigma_4\sigma_3\sigma_1)
(\sigma_5\sigma_4\sigma_6\sigma_5)\\
&=(\sigma_3\cdots\sigma_6)(\sigma_2\cdots\sigma_5)\sigma_1
\end{align*}
and similarly,
\begin{align*}
\rho_{\check A}&=(\rho_3\rho_5\rho_7)(\rho_3\rho_5)\rho_3\\
&=(\sigma_3\cdots\sigma_8)(\sigma_2\cdots\sigma_7)
(\sigma_3\cdots\sigma_6)(\sigma_1\cdots\sigma_5)
(\sigma_3\sigma_4)(\sigma_1\sigma_2\sigma_3)\sigma_1\\
&=(\sigma_3\sigma_5\sigma_7)
(\sigma_4\dots\sigma_8)(\sigma_2\dots\sigma_7)
(\sigma_4\sigma_5\sigma_6)(\sigma_1\dots\sigma_5)
\sigma_4(\sigma_1\sigma_2\sigma_3)\sigma_1
.
\end{align*}
\end{exam}

In general, for any $n\ge 10$ and $A\in\mathcal{M}_{k,n}^0$, we have
\begin{align*}
\rho_{\check A}&=(\rho_3\rho_5\cdots\rho_{2p-7})(\rho_3\rho_5\cdots\rho_{2p-9})\cdots(\rho_3\rho_5)\rho_3\\
&=\sigma_1^{-1}\left(\prod_{i=4}^{p-1} (\sigma_3\cdots\sigma_{2(p-i)+2})
(\sigma_1\cdots\sigma_{2(p-i)+1})\right)\sigma_1\\
&=(\sigma_3\sigma_5\cdots\sigma_{2p-7}) \sigma_1^{-1}\left( \prod_{i=4}^{p-1}(\sigma_4\cdots\sigma_{2(p-i)+2})(\sigma_1\cdots\sigma_{2(p-i)+1}) \right)\sigma_1.\tag{$\ast$}
\end{align*}
where $p=\left\lfloor \frac{n-2}2 \right\rfloor$.
Therefore $S(\rho_{\check A})=\{3,5,\dots,2p-7\}$ and $\rho_{\check A}$ is the same for any terminal element in $\mathcal{M}_{k,n}^0$.

In the proof of the following lemma, we need to consider $\tau(\mathcal{M}_{k,n}^0)$. To avoid any unnecessary complication, we need $\mathcal{M}_{k,n}^0\cap\tau(\mathcal{M}_{k,n}^0)=\emptyset$ and this is guaranteed by assuming $n\ge 14$ from now on.

\begin{lem}
\label{lem:TwistedDecycling}
Let $A$ be an element of $\mathcal{M}_{k,n}^0$ such that $A=\check A$.
Then there exists a unique add-tail minimal $R$-conjugator
\begin{displaymath}
\rho=
\alpha(A_k)^{-1}\rho_A^{-1}\Delta
\end{displaymath}
where $A_k$ is the last row of $A$.
\end{lem}
\begin{proof}
Suppose $n$ is even. By hypothesis, $A=\check A=(\mathbf{0}\cdots\mathbf{0} A^{p-2}A^{p-1}\mathbf{1})$ for some $A^{p-2}, A^{p-1}\neq\mathbf{0},\mathbf{1}$.
Then $\hat A=(\mathbf{0}\mathbf{1}\cdots\mathbf{1}A^{p-2}A^{p-1}\mathbf{1})$.
Let $\rho$ be an add-tail minimal $R$-conjugator for $\alpha(A)$. Then $2i+1\in S(\rho)$ for some $1\le i\le p-1$.
By a repeated application of Lemma~\ref{lem:Join}(ii) and (iii), we eventually find $\sigma_{2p-1}\prec\rho$ whatever $i$ is.

Let $\gamma$ be a minimal element among all add-tail $R$-conjugators for $\alpha(\hat A)$ starting with $\sigma_{2p-1}$. Since no add-tail minimal $R$-conjugators for braids in $\Gamma(A)$ except $\check A$ start with $\sigma_{2p-1}$ by Lemma~\ref{lem:AddtailMinimalConjugator}, the conjugation of $\alpha(\hat A)$ by some prefix of $\gamma$ must first reach $\check A$ and therefore $\rho_A\prec\gamma$.
Let $\gamma=\rho_A \gamma'$. Then $\gamma'$ is an add-tail $R$-conjugator for $\alpha(A)$.
Since $\sigma_{2p-1}$ and $\rho_A$ commute by $(\ast)$, $2p-1\in S(\gamma')$ and so $\rho\prec \gamma'$ by the minimality of $\rho$.

On the other hand, since $\rho_A\rho$ is an add-tail $R$-conjugator for $\alpha(\hat A)$ starting with $\sigma_{2p-1}$, $\gamma\prec\rho_A\rho$ by the minimality of $\gamma$ and so $\gamma'\prec\rho$. Therefore $\gamma'=\rho$ and $\gamma=\rho_A\rho$.

But any $SSS$-conjugator for $\alpha(\hat A)$ starting with $\sigma_{2p-1}$ must contain $\sigma_{2p-5}\sigma_{2p-3}$ as a prefix by Lemma~\ref{lem:Join} and so $S(\rho_A\rho)=\{3,5,\dots,2p-1\}$
because $S(\rho_A)=\{3,5,\dots,2p-7\}$ as we discussed just before this lemma.
Since $\rho_A\rho$ is an add-tail $SSS$-conjugator, $\rho_A\rho\prec\alpha(\hat A)\rho_A\rho$. Since $S(\alpha(\hat A_1))=\{2,4,\dots,2p\}$, we have $\{2,3,\dots,2p\}\subset S(\alpha(\hat A)\rho_A\rho)$ and so
$$\Delta\prec \sigma_1\sigma_{2p+1}\alpha(\hat A)\rho_A\rho.$$

Then by considering the right-weighted form for $\sigma_1\sigma_{2p+1}\alpha(\hat A)\rho_A\rho$, there exists $\beta\prec\sigma_1\sigma_{2p+1}$ such that $\alpha(\hat A_k)\rho_A\rho\beta=\Delta$, that is, $\rho=\rho_A^{-1}\alpha(\hat A_k)^{-1}\Delta\beta^{-1}$. Since $\rho_A^{-1}\alpha(\hat A)\rho_A=\alpha(A)$, $\rho_A^{-1}\alpha(\hat A_k)\rho_A=\alpha(A_k)$.
Thus
$
\rho=\rho_A^{-1}\alpha(\hat A_k)^{-1}\Delta\beta^{-1}
=\alpha(A_k)^{-1}\rho_A^{-1}\Delta \beta^{-1}$
and so
\begin{displaymath}
\rho^{-1}\alpha(A)\rho=\beta\Delta^{-1}\rho_A\alpha(A_k)\alpha(A)
\alpha(A_k)^{-1}\rho_A^{-1}\Delta\beta^{-1}
=\beta\alpha(\tau\mathbf{c}^{-1}(\hat A))\beta^{-1},
\end{displaymath}
since $\alpha(A_k)\alpha(A)\alpha(A_k)^{-1}$ is again a terminal element and is sent to its initial element via the conjugation by $\rho_A^{-1}$. If we consider the involuted versions of Lemma~\ref{lem:Join} and \ref{lem:AddtailMinimalConjugator} for $\tau(\mathcal{M}_{k,n}^0)$, an initial element $\alpha(\tau\mathbf{c}^{-1}(\hat A))$ cannot have an $R$-conjugator starting with $\sigma_1$ or $\sigma_{2p+1}$. Thus $\beta=e$ and so we have $\rho=\alpha(A_k)^{-1}\rho_A^{-1}\Delta$
as desired and $\rho^{-1}\alpha(A)\rho=\alpha(\tau\mathbf{c}^{-1}(\hat A))$.

Suppose $n$ is odd. Assume $\alpha(A,p-2)=[\mathbf{0}\dots\mathbf{0}A^{p-2}|A^{p-1}\mathbf{1}]$.
Let $\rho$ be an add-tail minimal $R$-conjugator for $\alpha(A,p-2)$,
and $\rho'\in\B_{n-1}$ be the braid obtained by removing the $(2p-2)$-nd strand from $\rho$.
Then $\rho'$ is nontrivial. Indeed, if $\rho'$ were trivial then $S(\gamma)$ consists of either $2p-3$ or $2p-2$. But Lemma~\ref{lem:JoinOdd} forces $S(\rho)$ to contain $2i+1$ for $1\le i\le p-3$ or $2p$ under the assumption that $\rho$ is an add-tail $SSS$-conjugator and this is a contradiction. Thus $\rho'$ is an add-tail $R$-conjugator for $\alpha(A)\in\B_{n-1}$.
The above argument for even braids can be applied to $\alpha(A)$ and $\rho'$ to obtain an add-tail $R$-conjugator $\rho'=\alpha(A_k)^{-1}\rho_A^{-1}\Delta\beta'^{-1}$ for some $\beta'\prec\sigma_1\sigma_{2p+1}\in\B_{n-1}$.
In addition we have $\sigma_{2p}\prec\rho$ since $\sigma_{2p-1}\prec\rho'$.

Let $\bar\rho_A$ be the $n$-braid obtained from $\rho_A$ by inserting a vertical strand between the $(2p-3)$-rd and the $(2p-2)$-nd punctures. Note that $\bar\rho_A$ and $\rho_A$ has the identical braid word. Clearly $\alpha(A_k,p-2)^{-1}\bar\rho_A^{-1}\Delta$ is an add-tail $R$-conjugator and contains $\sigma_{2p}$ as a prefix where $\Delta$ is understood to be the half twist of the braid index being considered. By the minimality of $\rho$, $$\rho\beta=\alpha(A_k,p-2)^{-1}\bar\rho_{A}^{-1}\Delta$$ for some permutation braid $\beta$.
By removing the $(2p-2)$-nd strand in both sides of this equation, we have $\rho'\beta''=\alpha(A_k)^{-1}\rho_A^{-1}\Delta$ where $\beta''$ is the braid obtained from $\beta$.
Thus $\rho'\beta'=\rho'\beta''$ and so $\beta'=\beta''$.

Let $\rho(2p-2)=m$. Since the $(2p-2)$-nd strand is vertical in $\bar\rho_{A}\alpha(A_k,p-2)$, we have $(\rho\beta)(2p-2)=\beta(m)=\Delta(2p-2)=5$.
Since $\beta'\prec\sigma_1\sigma_{2p+1}$, $\beta=\bar\beta\sigma_1^{\epsilon_1}\sigma_{n-1}^{\epsilon_2}$ where $\epsilon_i$'s are either 0 or 1, and $\bar\beta$ becomes trivial when the $m$-th strand is removed.

Now we have
\begin{align*}
\rho^{-1}\alpha(A,p-2)\rho
&=\bar\beta\left\{(\sigma_1^{\epsilon_1}\sigma_{n-1}^{\epsilon_2})\alpha(\tau\mathbf{c}^{-1}(\hat A,p-2))(\sigma_1^{\epsilon_1}\sigma_{n-1}^{\epsilon_2})^{-1}\right\}\bar\beta^{-1}.
\end{align*}
Let $W$ be the reversed word of the part inside the braces in the right hand side. Then $\bar\beta^{-1}W\bar\beta$ is obviously in its super summit set. If $\bar\beta$ were nontrivial, the braid obtained by removing the $m$-th strand is also nontrivial by the same reason why $\rho'$ was nontrivial in the above. But this contradicts the assumption of $\bar\beta$. Thus $\bar\beta=e$ and $\beta\prec\sigma_1\sigma_{n-1}$. This implies $\beta=e$, otherwise $\rho$ cannot be an $R$-conjugator.
Consequently, we have $\rho=\alpha(A_k,p-2)^{-1}\bar\rho_{A}^{-1}\Delta$ and the notations for the vertical strand can be suppressed to obtain the desired add-tail minimal $R$-conjugator.
\end{proof}

We call this type of conjugation {\em initializing} because it convert a terminal element from $\mathcal{M}_{k,n}^0$ or $\tau(\mathcal{M}_{k,n}^0)$ to an initial element in $\tau(\mathcal{M}_{k,n}^0)$ or $\mathcal{M}_{k,n}^0$, respectively.

\section{Main theorem}

We summarize our results in the following theorem.

\begin{thm}
\label{thm:pA}
For $A\in\mathcal{M}_{k,n}^0$, the rigid conjugacy set $R(\alpha(A))$ is contained in $\alpha(\mathcal{M}_{k,n}^0\cup\tau(\mathcal{M}_{k,n}^0))$.
\end{thm}
\begin{proof}
By Lemma~\ref{lem:CutheadMinimalConjugator},
the conjugation of $\alpha(A)$ by a cut-head minimal $R$-conjugator is cycling. By Lemma~\ref{lem:AddtailMinimalConjugator} and Lemma~\ref{lem:TwistedDecycling}, the conjugation by an add-tail minimal $R$-conjugator is either switching or initializing. Thus $\alpha(\mathcal{M}_{k,n}^0\cup\tau(\mathcal{M}_{k,n}^0))$ is closed under all of these operations and so includes the rigid conjugacy set $R(\alpha(A))$.
\end{proof}

Figure~\ref{fig:RSSS} displays $R(\alpha(A))$ for $A\in \mathcal{M}_{5,n}^0\cup\tau(\mathcal{M}_{5,n}^0)$ as a directed graph. The ordinary arrows and the dashed arrows indicates cycling and initializing, respectively. All switching occurs inside lattices indicated by shaded rhombi.

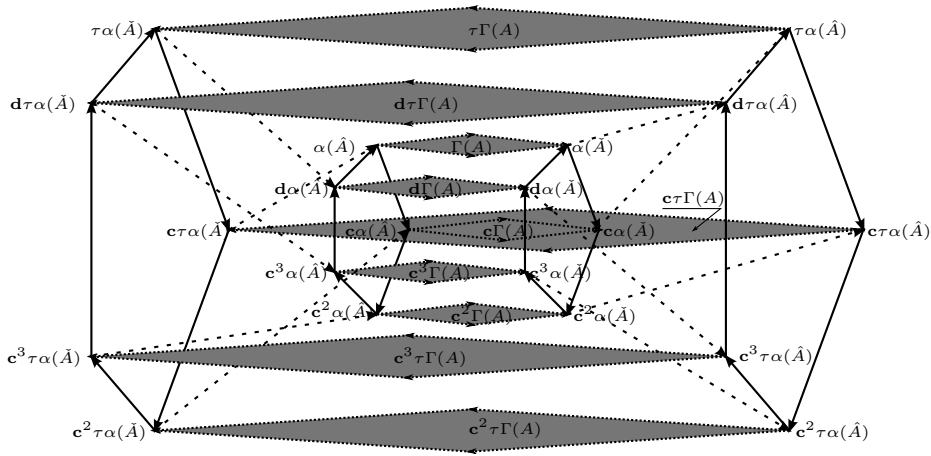
\begin{figure}[ht]
\begin{center}
\begingroup
   \setlength{\unitlength}{334.44574023pt}
\begin{picture}(1,0.50745522)%
    \put(0,0){\includegraphics[width=\unitlength]{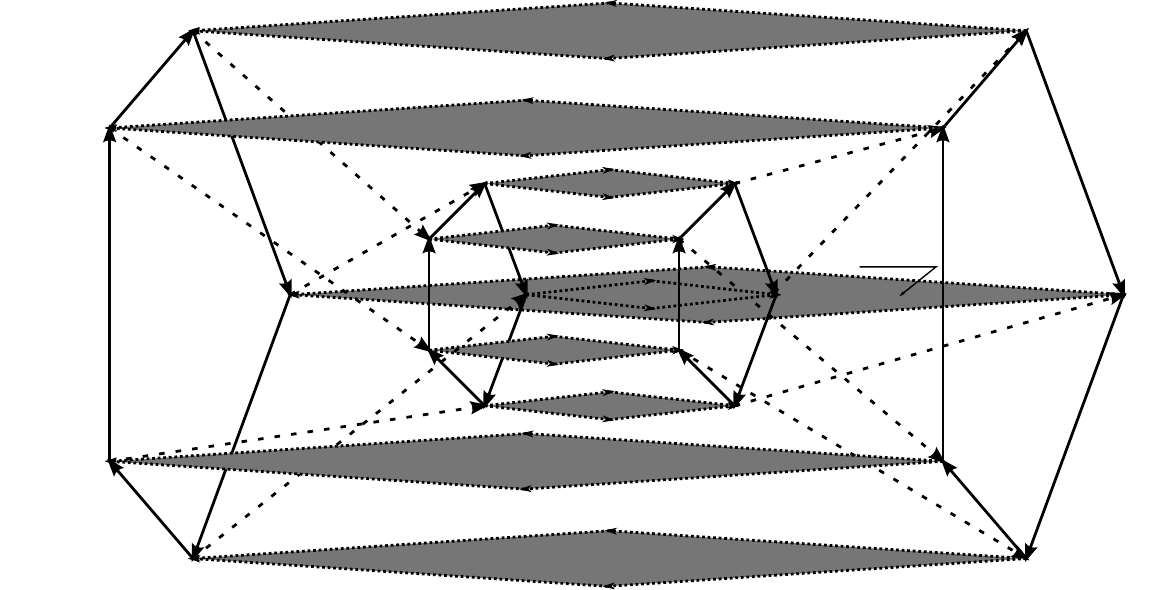}}%
    \put(0.34537868,0.34223331){\color[rgb]{0,0,0}\makebox(0,0)[lb]{\smash{\tiny$\alpha(\hat A)$}}}%
    \put(0.38125894,0.24655262){\color[rgb]{0,0,0}\makebox(0,0)[lb]{\smash{\tiny$\mathbf{c}\alpha(\hat A)$}}}%
    \put(0.3431242,0.15540249){\color[rgb]{0,0,0}\makebox(0,0)[lb]{\smash{\tiny$\mathbf{c}^2\alpha(\hat A)$}}}%
    \put(0.29107045,0.19871227){\color[rgb]{0,0,0}\makebox(0,0)[lb]{\smash{\tiny$\mathbf{c}^3\alpha(\hat A)$}}}%
    \put(0.30168166,0.29439296){\color[rgb]{0,0,0}\makebox(0,0)[lb]{\smash{\tiny$\mathbf{d}\alpha(\hat A)$}}}%
    \put(0.63741839,0.34223331){\color[rgb]{0,0,0}\makebox(0,0)[lb]{\smash{\tiny$\alpha(\check A)$}}}%
    \put(0.67329865,0.24655262){\color[rgb]{0,0,0}\makebox(0,0)[lb]{\smash{\tiny$\mathbf{c}\alpha(\check A)$}}}%
    \put(0.63981041,0.15087193){\color[rgb]{0,0,0}\makebox(0,0)[lb]{\smash{\tiny$\mathbf{c}^2\alpha(\check A)$}}}%
    \put(0.58957804,0.19871227){\color[rgb]{0,0,0}\makebox(0,0)[lb]{\smash{\tiny$\mathbf{c}^3\alpha(\check A)$}}}%
    \put(0.58957804,0.29439296){\color[rgb]{0,0,0}\makebox(0,0)[lb]{\smash{\tiny$\mathbf{d}\alpha(\check A)$}}}%
    \put(0.81921171,0.39007366){\color[rgb]{0,0,0}\makebox(0,0)[lb]{\smash{\tiny$\mathbf{d}\tau\alpha(\hat A)$}}}%
    \put(0.88858021,0.47379426){\color[rgb]{0,0,0}\makebox(0,0)[lb]{\smash{\tiny$\tau\alpha(\hat A)$}}}%
    \put(0.97230081,0.24655262){\color[rgb]{0,0,0}\makebox(0,0)[lb]{\smash{\tiny$\mathbf{c}\tau\alpha(\hat A)$}}}%
    \put(0.89097222,0.01931097){\color[rgb]{0,0,0}\makebox(0,0)[lb]{\smash{\tiny$\mathbf{c}^2\tau\alpha(\hat A)$}}}%
    \put(0.82706475,0.10514584){\color[rgb]{0,0,0}\makebox(0,0)[lb]{\smash{\tiny$\mathbf{c}^3\tau\alpha(\hat A)$}}}%
    \put(0.09336256,0.47379426){\color[rgb]{0,0,0}\makebox(0,0)[lb]{\smash{\tiny$\tau\alpha(\check A)$}}}%
    \put(0.17941839,0.24655262){\color[rgb]{0,0,0}\makebox(0,0)[lb]{\smash{\tiny$\mathbf{c}\tau\alpha(\check A)$}}}%
    \put(0.07311086,0.01931097){\color[rgb]{0,0,0}\makebox(0,0)[lb]{\smash{\tiny$\mathbf{c}^2\tau\alpha(\check A)$}}}%
    \put(-0.00083721,0.10303158){\color[rgb]{0,0,0}\makebox(0,0)[lb]{\smash{\tiny$\mathbf{c}^3\tau\alpha(\check A)$}}}%
    \put(0.00259409,0.39007366){\color[rgb]{0,0,0}\makebox(0,0)[lb]{\smash{\tiny$\mathbf{d}\tau\alpha(\check A)$}}}%
    \put(0.5010734,0.34223331){\color[rgb]{0,0,0}\makebox(0,0)[lb]{\smash{\tiny$\Gamma( A)$}}}%
    \put(0.53695366,0.24655262){\color[rgb]{0,0,0}\makebox(0,0)[lb]{\smash{\tiny$\mathbf{c}\Gamma( A)$}}}%
    \put(0.45323306,0.19871227){\color[rgb]{0,0,0}\makebox(0,0)[lb]{\smash{\tiny$\mathbf{c}^3\Gamma( A)$}}}%
    \put(0.5010734,0.15087193){\color[rgb]{0,0,0}\makebox(0,0)[lb]{\smash{\tiny$\mathbf{c}^2\Gamma( A)$}}}%
    \put(0.45323306,0.29439296){\color[rgb]{0,0,0}\makebox(0,0)[lb]{\smash{\tiny$\mathbf{d}\Gamma( A)$}}}%
    \put(0.52020954,0.47379426){\color[rgb]{0,0,0}\makebox(0,0)[lb]{\smash{\tiny$\tau\Gamma( A)$}}}%
    \put(0.7401033,0.28379736){\color[rgb]{0,0,0}\makebox(0,0)[lb]{\smash{\tiny$\mathbf{c}\tau\Gamma( A)$}}}%
    \put(0.43648894,0.10303158){\color[rgb]{0,0,0}\makebox(0,0)[lb]{\smash{\tiny$\mathbf{c}^3\tau\Gamma( A)$}}}%
    \put(0.52130262,0.02272957){\color[rgb]{0,0,0}\makebox(0,0)[lb]{\smash{\tiny$\mathbf{c}^2\tau\Gamma( A)$}}}%
    \put(0.43648894,0.39007366){\color[rgb]{0,0,0}\makebox(0,0)[lb]{\smash{\tiny$\mathbf{d}\tau\Gamma( A)$}}}%
  \end{picture}%
\endgroup
\end{center}
\caption{The directed graph of $R(\alpha( A))$}
\label{fig:RSSS}
\end{figure}

Finally we give the main theorem.

\begin{thm}\label{thm:SizeofRSSS}
For $n\ge 14$, let $ A\in \mathcal{M}_{k,n}^0$ be given arbitrarily. Then the rigid conjugacy set $R(\alpha(A))$ of $\alpha(A)$ is made of pseudo-Anosov braids and its size is given by
\begin{displaymath}
|R(\alpha( A))|
=k\cdot 2^{p-3}\ge \frac 1{23} k\sqrt{2}^n
\end{displaymath}
where $p=\lfloor\frac{n-2}2\rfloor$.
\end{thm}
\begin{proof}
Since $R(\alpha(A))\subset\alpha(\mathcal{M}_{k,n}^0\cup\tau(\mathcal{M}_{k,n}^0))
\subset\alpha(\mathcal{M}_{k,n})$ and none of braids in $\alpha(\mathcal{M}_{k,n})$ has a standard reduction system by Lemma~\ref{lem:BinaryBraid}(2), all of braids in $R(\alpha(A))$ is pseudo-Anosov according to \cite{BNG}.

To compute the size of the rigid conjugacy set $R(\alpha(A))$, it is enough to consider the contributions of cycling, switching, and initializing.
By the definition of $\mathcal{M}_{k,n}^0$, it is easy to see that cyclings contribute the factor $k$.

Switching operates on constant columns from the 2nd to $(p-3)$rd for $\alpha(\mathcal{M}_{k,n}^0)$ and from 4th to $(p-1)$st for $\alpha(\tau(\mathcal{M}_{k,n}^0))$ and can generate any sequence of $\mathbf{0}$'s and $\mathbf{1}$'s. Since the length of sequences is $p-4$, switchings contribute the factor $2^{p-4}$.

Since $n\ge 14$, $\mathcal{M}_{k,n}^0\cap\tau(\mathcal{M}_{k,n}^0)=\emptyset$ and so an initialization contributes the factor 2.
By multiplying all these factors, we have the desired formula for $|R(\alpha( A))|$.
\end{proof}

\end{document}